\documentclass{amsart}

\usepackage{amsfonts}
\usepackage{amsmath}
\usepackage{amsthm}
\usepackage{amssymb}
\usepackage{ifsym}
\usepackage{dirtytalk}

\theoremstyle{definition}

\theoremstyle{remark}

\numberwithin{equation}{section}

\newcommand\blfootnote[1]{%
  \begingroup
  \renewcommand\thefootnote{}\footnote{#1}%
  \addtocounter{footnote}{-1}%
  \endgroup
}

\begin{document}

\title[]{Distinct Orders Dividing Each Other on Both Sides}
\author{Garrett Ervin}

\begin{abstract}
We construct non-isomorphic linear orders $X$ and $Y$ that are both left-hand and right-hand divisors of one another, answering positively a question of Sierpi\'nski. 
\end{abstract}

\maketitle

\section{Introduction}

\blfootnote{\emph{2010 Mathematics Subject Classification}: 06A05, 03E05.}
\blfootnote{\emph{Key words and phrases:} linear order, lexicographical product, Schroeder-Bernstein problem}

In his book \emph{Cardinal and Ordinal Numbers} \cite{Sierpinski}, Sierpi\'nski posed five questions concerning the class $(LO, \times)$ of linear orders under the lexicographical product. 
\begin{enumerate}
\item[Q1.] Does there exist a linear order $X$ that is isomorphic to its cube but not to its square?
\item[Q2.] Do there exist non-isomorphic linear orders $X$ and $Y$ that are left-hand and right-hand divisors of each other? 
\item[Q3.] Do there exist non-isomorphic orders $X$ and $Y$ whose cubes are isomorphic but squares are not? 
\item[Q4.] Do there exist non-isomorphic orders $X$ and $Y$ whose squares are isomorphic but cubes are not? 
\item[Q5.] Do there exist \emph{countable} non-isomorphic orders $X$ and $Y$ that are right-hand divisors of each other? 
\end{enumerate}

Though analogous questions have since been answered for many other classes of structures, these problems had all remained open until the author solved the first problem negatively in \cite{Ervin}. A negative answer to the fifth problem follows quickly from the work in that paper. The main result of this paper is a positive solution to the second problem. 

Before presenting the construction, we motivate these problems and give some historical background. Let $(K, \times)$ be a class of structures equipped with an associative product. Distinguishing the structures in $K$ only up to isomorphism, we may view $(K, \times)$ as a semigroup, and ask whether this semigroup possesses any familiar properties, such as cancellation ($A \times X \cong B \times X \Longrightarrow A \cong B$) or unique square roots ($X^2 \cong Y^2 \Longrightarrow X \cong Y$). 

While many natural classes of finite structures satisfy cancellation (see \cite{Lovasz}), classes that contain infinite structures frequently do not. For example, in any class $K$ that is closed under countable direct sums, there will exist infinite structures $A$ and $X$ such that $A \times X \cong X$. In many classes it is even possible to find infinite structures $X$ that are isomorphic to their own squares. 

Even when cancellation fails, one may ask if $(K, \times)$ satisfies other less stringent regularity properties. Two such properties that were considered historically are the \emph{Kaplansky test properties}. The first of these, sometimes individually referred to as the \emph{Schroeder-Bernstein property}, asserts that whenever two structures $X, Y \in K$ are each isomorphic to a divisor of the other, then $X$ and $Y$ must themselves be isomorphic. That is, if $X \cong A \times Y$ and $Y \cong B \times X$ for some $A, B \in K$, then $X \cong Y$. The second, the \emph{unique square root property}, asserts that for all $X, Y \in K$, if $X^2 \cong Y^2$ then $X \cong Y$. 

Kaplansky himself was interested in determining whether these properties held for various classes of infinite abelian groups. He considered their failure a heuristic indication that a given class admitted no useful structure theorem, a la the classification theorem for finitely generated abelian groups. In his book \cite{Kaplansky} in which he recorded them, he said ``I believe their defeat is convincing evidence that no reasonable invariants exist." The two properties had been considered previously for other classes of structures by several authors, including Tarski \cite{Tarski} and Hanf \cite{Hanf}. 

A third property, related to the two Kaplansky properties, is the \emph{cube property}. Whenever a structure $X$ from a class $K$ is isomorphic to its own square $X^2$, it is isomorphic to all of its finite powers $X^n$. In particular, it is isomorphic to its cube. The cube property asserts that the converse holds: for all $X \in K$, if $X \cong X^3$ then it is already the case that $X \cong X^2$. If $K$ contains a counterexample to the cube property, namely, a structure $X$ that is isomorphic to its cube but not to its square, then the pair $X$ and $Y = X^2$ witnesses the failure of both Kaplansky properties for $K$. 

The cube property holds trivially if $K$ contains no infinite structure isomorphic to its cube. Usually when $K$ does contain such structures, it is possible to find one that is not isomorphic to its square. The first example was produced by Hanf, who constructed in \cite{Hanf} a Boolean algebra isomorphic to its direct cube but not to its square. Subsequent to his result, the cube property has been shown to fail for a variety of algebraic, relational, and topological classes of structures. See the introduction to \cite{Ervin} for a detailed list. 

Counterexamples to the cube property can be recast in the language of representing semigroups. If $(S, \cdot)$ is a semigroup, then $S$ is \emph{represented in $K$} if there is a map $i: S \rightarrow K$ such that for all $x, y \in S$, we have $i(x \cdot y) \cong i(x) \times i(y)$, and if $x \neq y$ then $i(x) \not\cong i(y)$. The failure of the cube property for $K$ is equivalent to the statement that the group $\mathbb{Z}_2$ can be represented in $K$. More generally, $\mathbb{Z}_n$ can be represented in $K$ if and only if there is a structure $X \in K$ that is isomorphic to $X^{n+1}$ but whose lower powers $X, X^2, \ldots, X^n$ are pairwise non-isomorphic. 

If the cube property fails for $K$, then usually not only $\mathbb{Z}_2$ but every finite cyclic group can be represented in $K$. In some cases, more spectacular representation results hold. Ketonen showed in \cite{Ketonen} that every countable commutative semigroup can be represented in the class $(BA, \times)$ of countable Boolean algebras under the direct product. Trnkov\'a and Koubek showed in \cite{Trnkova} that every commutative semigroup can be represented in the class $(G, \times)$ of graphs, where the product can be taken to be the categorical product, strong product, or Cartesian product. It follows that these classes satisfy no general laws that cannot be derived solely from the associativity and commutativity of their products. 

If $X$ and $Y$ are linear orders, their lexicographical product $X \times Y$ is the order obtained by replacing every point in $X$ with a copy of $Y$. The lexicographical product is somewhat peculiar among the many natural products appearing in mathematics in that it is not commutative. For non-commutative products, ``asymmetric" properties like cancellation and the Schroeder-Bernstein property have both left-sided and right-sided versions. 

Both left and right cancellation fail for general linear orders. Indeed for any fixed order $A$, there are many examples of orders $X$ and $Y$ such that $A \times X \cong X$ and $Y \times A \cong Y$. Moreover, there are orders $X$ of any infinite cardinality such that $X \cong X^2$. Certain special subclasses of orders, however, do possess cancellation laws. See \cite{Morel}. 

The left-sided and right-sided versions of the Schroeder-Bernstein property also fail for $(LO, \times)$. Sierpi\'nski himself was aware of counterexamples when he wrote \emph{Cardinal and Ordinal Numbers}. That is, he knew of non-isomorphic orders $X$ and $Y$ of the form $X \cong Y \times A$ and $Y \cong X \times B$, and of non-isomorphic orders $X'$ and $Y'$ of the form $X' \cong C \times Y'$ and $Y' \cong D \times X'$, though in the latter case he knew of only uncountable examples. (Here, $X$ and $Y$ are said to be \emph{left-hand divisors} of one another, and $X'$ and $Y'$ are \emph{right-hand divisors} of one another.) He was also aware of A.C. Davis's counterexamples to the unique square root property, which he, along with Davis, generalized in \cite{Davis}. All of their examples were pairs of non-isomorphic orders $X, Y$ with the property that not only $X^2 \cong Y^2$ but actually $X^n \cong Y^n$ for all $n > 1$. 

But while Sierpi\'nski was thus able to witness the failure of both Kaplansky properties for $(LO, \times)$, he was not able to adapt the constructions to prove the existence of some seemingly ``near by" counterexamples. He knew of no counterexample to the cube property [Q1], nor of a single pair of orders witnessing the failure of the left-sided and right-sided Schroeder-Bernstein properties simultaneously [Q2]. And he did not know if the ``collapse of powers" seen in Davis's counterexamples to the unique square root property is a necessary phenomenon [Q3] [Q4]. He wrote on page 232 of \emph{Cardinal and Ordinal Numbers}, ``We do not know so far any example of two [linear order] types $\varphi$ and $\psi$, such that $\varphi^2 = \psi^2$ but $\varphi^3 \neq \psi^3$ [Q4], or of types $\gamma$ and $\delta$ such that $\gamma^2 \neq \delta^2$ but $\gamma^3 = \delta^3$ [Q3]. Neither do we know of any type $\alpha$ such that $\alpha^2 \neq \alpha^3 = \alpha$ [Q1]." Then, on page 251, ``We do not know\ldots whether there exist two different order types which are both left-hand and right-hand divisors of each other [Q2]." 

A linear order $X$ of the desired type $\alpha$ is a counterexample to the cube property for $(LO, \times)$. If such an $X$ exists, then $X$ (which is isomorphic to $X^3$) and $X^2$ are non-isomorphic orders that are both left-hand and right-hand divisors of each other, yielding a positive answer to Q2 as well. As mentioned already, such an order gives a representation of $\mathbb{Z}_2$ in $(LO, \times)$. If it were also possible to represent $\mathbb{Z}_6$ and $\mathbb{Z}_4$ in $(LO, \times)$, then orders of the desired types $\varphi$ and $\psi$ would exist, as would orders of the desired types $\gamma$ and $\delta$, since elements satisfying these identities exist in these groups.

It turns out, however, there is no such order type $\alpha$: $(LO, \times)$ is one of the rare classes for which the cube property holds, as shown in \cite{Ervin}. It is the unique example, to the author's knowledge, of a natural class of structures for which the cube property holds but Schroeder-Bernstein properties fail. It was proved in \cite{Ervin} that more generally no finite cyclic group is represented in $(LO, \times)$. Hence the easy solutions to Q2, Q3, and Q4 described in the previous paragraph are not available. 

Despite this, we will show that Q2 does have a positive answer, by way of the following theorem. 

\theoremstyle{definition}
\newtheorem*{mt}{Main Theorem}
\begin{mt}
Let $\omega$ denote the natural numbers in their usual order, and let $A = \omega_1^* + \omega_1$ be the ordered sum of the first uncountable ordinal and its reverse. There exist non-isomorphic orders $X$ and $Y$ such that $X \cong A \times Y \cong Y \times \omega$ and $Y \cong A \times X \cong X \times \omega$. 
\end{mt}

This appears as Theorem \ref{sbbothsides} below. To the author's knowledge, both questions Q3 and Q4 remain open. The more general question of precisely what semigroups are represented in $(LO, \times)$ is wide open.

In the next section, we recall the relevant terminology and notation, and review some previous results that we will need for our construction. In the third section, we prove the main theorem. The proof is elementary, but makes use of a representation theorem and fixed-point theorem proved in \cite{Ervin}, as well as some basic facts about ordinals and their arithmetic. 

The work here is an elaboration of the fourth chapter of the author's thesis \cite{Thesis}. 

\section{Preliminaries}

\subsection{Terminology} 

A linear order is a pair $(X, <)$, where $X$ is a set and $<$ is an irreflexive, antisymmetric, and transitive binary relation on $X$. We will always refer to an order $(X, <)$ by its underlying set $X$. The empty order is considered a legitimate linear order. 

We use lowercase Greek letters to denote ordinals. As usual, $\omega$ denotes the first infinite ordinal and $\omega_1$ denotes the first uncountable ordinal. We take $\omega$ to include $0$. Ordinals will play two roles in what follows, as linear orders themselves, and as indexing sets. A \emph{sequence} is a set of points $\{x_i: i < \delta\}$ indexed by an ordinal $\delta$. 

If $X$ is a linear order, a subset $I \subseteq X$ is called an \emph{interval} if for all $x, y, z \in X$, if $x < y < z$ and $x, z \in I$, then $y \in I$. An interval $I$ is an \emph{initial segment} of $X$ if whenever $x \in I$ and $y < x$, then $y \in I$. An interval $J$ is a \emph{final segment} of $X$ if $X \setminus I$ is an initial segment of $X$, or equivalently if whenever $x \in J$ and $y > x$, then $y \in J$. If $I$ is an initial segment of $X$ and $J = X \setminus I$ is the corresponding final segment, the pair $(I, J)$ is called a \emph{cut}. 

If $I$ and $J$ are intervals in $X$, and every point in $I$ is less than every point in $J$, then we write $I < J$. If $C_1 = (I_1, J_1)$ and $C_2 = (I_2, J_2)$ are cuts in $X$, and $I_1$ is a strict initial segment of $I_2$, then the cut $C_1$ falls to the left of $C_2$ and we write $C_1 < C_2$. 

An order $X$ is \emph{dense} if it is infinite and whenever $x, y \in X$ and $x < y$ there is $z \in X$ with $x < z < y$. An order is \emph{scattered} if it contains no dense suborder. All ordinals are scattered. 

A strictly increasing sequence of points $\{x_i: i < \delta\}$ is \emph{cofinal} in $X$ if for every $y \in X$ there is an $i < \delta$ such that $y \leq x$. The \emph{cofinality} of $X$ is the shortest possible length $\lambda$ of a cofinal sequence in $X$. If $X$ has a right endpoint, then the cofinality of $X$ is 1. Otherwise, the cofinality of $X$ is an infinite regular cardinal. Similarly, a strictly decreasing sequence of points is \emph{coinitial} in $X$ if it eventually goes below every point in $X$, and the \emph{coinitiality} of $X$ is the shortest length $\kappa$ of a such sequence in $X$. Like the cofinality, the coinitiality of $X$ is either 1 or infinite and regular. 

If $I \subseteq X$ is an interval, then viewing $I$ as an order itself we may speak of the coinitiality and cofinality of $I$. If $I$ is an initial segment of $X$ and $J = X \setminus I$ is the corresponding final segment, then the cut $(I, J)$ is said to be a $(\kappa, \lambda)$-\emph{cut} if $I$ has cofinality $\kappa$ and $J$ has coinitiality $\lambda$. If $\kappa$ and $\lambda$ are both infinite, the cut $(I, J)$ is called a \emph{gap}. We will use ordinal notation rather than cardinal notation when referring to cuts, writing for example $(\omega, \omega_1)$-cut instead of $(\aleph_0, \aleph_1)$-cut. 

If $X$ is a linear order, then $X^*$ denotes the reverse order. For every point $x \in X$ there is a corresponding point $-x \in X^*$, and we have $x < y$ in $X$ if and only if $-x > -y$ in $X^*$. 

If $X$ and $Y$ are linear orders, their lexicographical product $X \times Y$ is the order obtained by replacing every point in $X$ with a copy of $Y$. Formally, we have $X \times Y = \{(x, y): x \in X, y \in Y\}$, ordered lexicographically. We usually omit the symbol $\times$ and write $XY$ for $X \times Y$. Every $x \in X$ determines an interval in $XY$ of order type $Y$, namely the set of points of the form $(x, \cdot)$ with first coordinate $x$. 

For a fixed $n \in \omega$, we may think of $n$ as a linear order by identifying $n$ with the order $0 < 1 < \ldots < n-1$. Thus, $nX$ denotes the order consisting of $n$ copies of $X$. 

The lexicographical product is associative, in the sense that $(XY)Z \cong X(YZ)$ for all orders $X, Y, Z$. Hence we may unambiguously define longer products as lexicographically ordered sets of tuples. In particular, for an order $X$ and $n \in \omega$, $X^n$ denotes the set of tuples $(x_0, \ldots, x_{n-1})$ of elements of $X$, ordered lexicographically, and $X^{\omega}$ denotes the set of sequences $(x_0, x_1, \ldots)$ of elements of $X$, ordered lexicographically. One may similarly define $X^{\delta}$ for any ordinal $\delta$. These powers obey the exponentiation rule $X^{\delta}X^{\gamma} \cong X^{\delta + \gamma}$, where $+$ is the ordinal sum.

We remark that if $X$ is a linear order without endpoints, or with only a single endpoint, then it is not hard to show that $X^{\omega}$ is dense. If $X$ has both endpoints (and they are distinct), then $X^{\omega}$ is never dense, but always contains a dense suborder. 

The notion of a \emph{replacement} generalizes the notion of a product by allowing points to be substituted by orders of more than one order type. Specifically, if $X$ is a linear order, and for every $x \in X$ we fix an order $I_x$, the replacement $X(I_x)$ is the order obtained by replacing every point $x$ with the corresponding $I_x$. Formally, $X(I_x) = \{(x, i): x \in X, i \in I_x\}$, ordered lexicographically. Every $y \in X$ determines an interval of order type $I_y$ in $X(I_x)$, namely the set of points $(y, \cdot)$ with first coordinate $y$. We allow that for a given $y$ we have $I_y = \emptyset$, and in $X(I_x)$ think of $y$ as being ``replaced by a gap."

We may also define the \emph{sum} of two linear orders. If $X$ and $Y$ are linear orders, $X + Y$ denotes the order obtained by putting a copy of $X$ to the left of a copy of $Y$. Formally, we view $X + Y$ as a replacement of the order $2 = \{0, 1\}$, where the left point is replaced by $X$ and the right point is replaced by $Y$. This operation agrees with the traditional ordinal sum when $X$ and $Y$ are ordinals. 

The notion of a sum is closely related to the notion of a cut. If $(I, J)$ is a cut in $X$, then $X$ is isomorphic to $I+J$. Conversely, if $X$ and $Y$ are orders, then $(X, Y)$ is a cut in $X + Y$. We will informally refer to this cut as ``the cut at the + sign."

One may also define infinite sums of orders. Let $\mathbb{Z}$ denote the integers in their usual order. Given orders $X_i, i \in \mathbb{Z}$, we write $\ldots + X_{-1} + X_0 + X_1 + \ldots$ to denote the replacement $\mathbb{Z}(X_i)$. Similarly, we write $X_0 + X_1 + \ldots$ to denote $\omega(X_i)$ and $\ldots + X_{-1} + X_0$ to denote $\omega^*(X_i)$. It is not hard to check that if the $X_i$ are scattered, then the sums $\mathbb{Z}(X_i)$, $\omega(X_i)$, and $\omega^*(X_i)$ are scattered. In particular, if the $X_i$ are ordinals, these sums are scattered. 

The lexicographical product is right-distributive over the sum, but not left-distributive. That is $(X+Y)Z \cong XZ + YZ$ for all orders $X, Y, Z$, but it is usually not the case that $Z(X+Y) \cong ZX + ZY$. More generally, the product distributes on the right over any replacement: we always have $X(I_x) \times Z \cong X(I_x \times Z)$. In particular, we have right-distributivity over the sums of type $\omega$, $\omega^*$, and $\mathbb{Z}$ described above.

A word of warning on exponential notation: there will be a place in our construction where the ordinal $\omega^{\omega}$ appears. Here, $\omega^{\omega}$ has its traditional meaning as $\sup_{n<\omega}\omega^n$, and not as the set of infinite sequences with entries from $\omega$, as is otherwise our convention in this paper. We will clearly point out when this ordinal appears, to avoid any confusion that might arise from the ambiguity of notation. We also note that while traditional ordinal exponents behave as expected with respect to the \emph{anti}-lexicographical product (the product usually used when studying ordinals), with regard to the lexicographical product there is some awkwardness. Namely, if $\alpha, \gamma, \delta$ are ordinals and if $\alpha^{\gamma}$ and $\alpha^{\delta}$ have their traditional meanings, then $\alpha^{\gamma} \times \alpha^{\delta} = \alpha^{\delta + \gamma}$ (note the reversal in the exponent). As we have noted above, for sets of $\delta$-sequences $X^{\delta}$, exponents behave as expected with respect to the lexicographical ordering. When there is no word to the contrary, $X^{\delta}$ is always assumed to mean the set of $\delta$-length sequences on $X$, ordered lexicographically. 

\subsection{Invariance under left multiplication by a given order}

Let $A$ be a fixed, nonempty linear order. We say that an order $X$ is invariant under left multiplication by $A$ if $AX \cong X$. A simple example of such an order is $X = A^{\omega}$. For, the natural map defined by $(a, (a_0, a_1, \ldots)) \mapsto (a, a_0, a_1, \ldots)$ is a bijection of $A \times A^{\omega}$ with $A^{\omega}$ which is also order-preserving, hence an isomorphism. We will refer to maps of this kind as flattening maps. 

The orders $X$ that are invariant under left multiplication by $A$ are characterized in \cite{Ervin}, and we will need to recall their general form. Before we can write it down, we need some more terminology. We denote the (unordered) set of finite sequences on $A$ as $A^{<\omega}$. This set includes the empty sequence. The length of a finite sequence $r$ is denoted $|r|$. If $r \in A^{<\omega}$ and $u \in A^{\omega}$, then $ru$ denotes the sequence beginning with $r$ and ending with $u$. We do not distinguish between elements of $A$ and sequences in $A^{<\omega}$ of length 1. 

If $u,v \in A^{\omega}$, then $u$ and $v$ are \emph{tail-equivalent} if there exist finite sequences $r, s \in A^{<\omega}$ and an infinite sequence $u' \in A^{\omega}$ such that $u = ru'$ and $v = su'$. If $u$ and $v$ are tail-equivalent we write $u \sim v$. It is easy to check that $\sim$ is an equivalence relation on $A^{\omega}$. It is the finest equivalence relation for which $au$ is equivalent to $u$ for all $a \in A$ and $u \in A^{\omega}$. We denote the tail-equivalence class of $u$ by $[u]$. 

Using tail-equivalence, we can construct many examples of orders $X$ such that $AX \cong X$, as follows. For every $u \in A^{\omega}$, fix an order $I_u$ such that whenever $u \sim v$ we have $I_u = I_v$. (Modulo this restriction the orders may be chosen arbitrarily.) Let $X = A^{\omega}(I_u)$ be the replacement of $A^{\omega}$ by these orders $I_u$. Then we have that $A \times A^{\omega}(I_u) \cong A^{\omega}(I_u)$, that is $AX \cong X$. To see this, note that for a fixed $a \in A$ and $u \in A^{\omega}$, we have in $A \times A^{\omega}(I_u)$ that the interval of points of the form $(a, u, \cdot)$ is of type $I_u$, and in $A^{\omega}(I_u)$ the interval of points $(au, \cdot)$ is of type $I_{au}$. But $I_u$ and $I_{au}$ are equal since $u \sim au$. Hence the flattening map $(a, u, x) \mapsto (au, x)$ is well-defined. Once we know this map is well-defined, it is easy to see that it is an order-isomorphism of $A \times A^{\omega}(I_u)$ with $A^{\omega}(I_u)$, that is, of $AX$ with $X$.  

We will denote orders $X$ of this form by $A^{\omega}(I_{[u]})$, thinking of $I_{[u]}$ as denoting the single order type of all $I_v$ with $v \in [u]$. We say that $X$ is a replacement of $A^{\omega}$ up to tail-equivalence. We have just argued that for such an $X$ we have $AX \cong X$. It turns out the converse is also true: if $AX \cong X$ then $X$ is isomorphic to an order of the form $A^{\omega}(I_{[u]})$. This is Theorem 3.5 of \cite{Ervin}.

In our proof of the main theorem, we will actually be interested in orders that are invariant under left multiplication by \emph{two} factors of $A$, that is, orders $X$ satisfying $A^2X \cong X$. A very similar construction can be used to build such orders, using a finer equivalence relation. 

For $u, v \in A^{\omega}$ we say that $u$ and $v$ are \emph{2-tail-equivalent} and write $u \sim_2 v$ if there exist finite sequences $r, s \in A^{<\omega}$ with $|r| \equiv |s| \pmod2$ and an infinite sequence $u' \in A^{\omega}$ such that $u = ru'$ and $v=su'$. One may check that this defines an equivalence relation. It is the finest equivalence relation for which $abu$ is equivalent to $u$ for all $a, b \in A$ and $u \in A^{\omega}$. The equivalence class of $u$ under this relation is denoted $[u]_2$. 

If $u \sim_2 v$, then $u \sim v$ as well. Hence the 2-tail-equivalence relation refines the tail-equivalence relation. On the other hand, for a given $u \in A^{\omega}$ and $a \in A$, observe that for any $v \in [u]$, either $v \sim_2 u$ or $v \sim_2 au$. Hence $[u] = [u]_2 \cup [au]_2$. For ``most" $u \in A^{\omega}$ we have that $u \not\sim_2 au$, so that the classes $[u]_2$ and $[au]_2$ are disjoint. But for certain $u$ whose entries eventually form repeating blocks, we have $u \sim_2 au$, in which case $[u]_2 = [au]_2 = [u]$. Specifically, we have that $u \sim_2 au$ if and only if there exists $r, s \in A^{<\omega}$ with $|s| \equiv 1 \pmod 2$ such that $u = rsss\ldots$ (see Proposition 3.7 of \cite{Ervin}). We say that such a sequence is eventually periodic, of odd period.

Whereas replacements of $A^{\omega}$ up to tail-equivalence are invariant under left multiplication by $A$, replacements up to 2-tail-equivalence are invariant under left multiplication by $A^2$. More precisely, suppose that for every $u \in A^{\omega}$ we fix an order $I_u$, such that if $u \sim_2 v$ then $I_u = I_v$. Let $X = A^{\omega}(I_u)$. Then $A^2X \cong X$, as witnessed by the map $(a, b, u, x) \mapsto (abu, x)$. This map is well-defined since $I_u = I_{abu}$ for all $a, b \in A, u \in A^{\omega}$. Conversely, it can be shown that any order invariant under left multiplication by $A^2$ is of this form. This is Theorem 3.10 of \cite{Ervin}. We write $X = A^{\omega}(I_{[u]_2})$. 

Replacements of the form $A^{\omega}(I_{[u]_2})$ are ``finer" than replacements of the form $A^{\omega}(I_{[u]})$, in the sense that in replacements of the former type, there may be tail-equivalent sequences $u,v$ which are not 2-tail-equivalent that are replaced by distinct orders. This is reflected in the fact that while an order $X$ that is invariant under left multiplication by $A$ is necessarily invariant under left multiplication by $A^2$, the converse need not hold: there are orders $A$ and $X$ such that $A^2X \cong X$ but $AX \not\cong X$. Indeed, in our proof of the main theorem we will construct such orders. 

If $X$ is isomorphic to $A^2X$, so that $X$ is of the form $A^{\omega}(I_{[u]_2})$, then while it need not be true that $X$ is isomorphic to $AX$, these orders do have closely related representations. Specifically, $AX$ is isomorphic to the order $A^{\omega}(J_{[u]_2})$, where for all $a \in A$ and $u \in A^{\omega}$ we have $J_{[u]_2} = I_{[au]_2}$. It follows that we always have $I_{[u]_2} = J_{[au]_2}$ as well. That is, $AX$ is the order obtained by interchanging the roles of $I_{[u]_2}$ and $I_{[au]_2}$ in the replacement $A^{\omega}(I_{[u]_2})$. This is Proposition 3.11 in \cite{Ervin}.

\subsection{A fixed-point theorem}

The last result we will need for our construction is a fixed-point theorem from \cite{Ervin}. For the remainder of the paper, let $A$ denote the order $\omega_1^* + \omega_1$. For simplicity, we identify the top point of $\omega_1^*$ and the bottom point of $\omega_1$, and write
\[
A = \ldots < -\alpha < \ldots < -\omega < \ldots < -1 < 0 < 1 < \ldots < \omega < \ldots < \alpha < \ldots 
\]

The proof of Theorem 6.1 from \cite{Ervin} establishes the following result.

\theoremstyle{definition}
\newtheorem{t4}[theorem]{Theorem}
\begin{t4}\label{fixedpoint}
Suppose that $f: A^{\omega} \rightarrow A^{\omega}$ is an order-automorphism. Then $f$ has a fixed point of the form $u = (\alpha_0, -\alpha_1, \alpha_2, -\alpha_3, \ldots)$, where for all $i \in \omega$ we have $\alpha_i \neq 0$. 
\end{t4}

The appearance of $\omega_1^* + \omega_1$ in this theorem is not accidental. If $X$ is a countable linear order, then there are always automorphisms of $X^{\omega}$ without fixed points. Even if $X$ is not countable, but has either countable cofinality or coinitiality, then it is usually possible to produce a fixed-point-free automorphism of $X^{\omega}$. In this sense, $A = \omega_1^* + \omega_1$ is the simplest candidate for an order for which we can guarantee that every automorphism $f: A^{\omega} \rightarrow A^{\omega}$ has a fixed point. The theorem tells us that indeed this holds. 

We note that a fixed point $u$ of the form guaranteed by the theorem need not be eventually periodic. But even if $u$ is eventually periodic, its only possible period is 2. Hence by the remarks in Section 2.3, for any $a \in A$ we have that the equivalence classes $[u]_2$ and $[au]_2$ are disjoint. We will use the existence of such a fixed point to show that the orders $X$ and $Y$ we construct in the proof of the main theorem are not isomorphic. 

\section{Proof of the main theorem}

We are now ready to prove the main theorem. First, we give a brief overview of the construction. Sierpi\'nski's Q2 from the introduction asks if there are non-isomorphic orders $X$ and $Y$ such that $X \cong AY$, $Y \cong A'X$ and $X \cong YB$, $Y \cong XB'$ for some orders $A, A', B, B'$. In Section 7 of \cite{Ervin} it is shown that the identities $X \cong AY$ and $Y \cong A'X$ already necessitate that at least one of the orders $A$ and $A'$ is uncountable (so that both $X$ and $Y$ must be uncountable as well). We will prove the stronger statement that there are such orders $X$ and $Y$, but with the left-hand divisors $A$ and $A'$ actually equal (to $\omega_1^* + \omega_1$) and the right-hand divisors $B$ and $B'$ equal as well (to $\omega$). 

It follows from the isomorphisms $X \cong AY$ and $Y \cong AX$ that $X \cong A^2X$. Hence, our $X$ will be of the form $A^{\omega}(I_{[u]_2})$ for some collection of orders $I_{[u]_2}$. As noted in Section 2.2, it must then be that $Y$ will be of the form $A^{\omega}(J_{[u]_2})$, where for every $u \in A^{\omega}$ and $a \in A$ we have $J_{[u]_2} = I_{[au]_2}$. Since for such $X, Y$ we automatically have that $Y \cong AX$ and $X \cong AY$, it remains only to specify the orders $I_{[u]_2}$, show that $X \omega \cong Y$ and $Y \omega \cong X$, and prove $X \not\cong Y$. 

We will build the $I_{[u]_2}$ so that the interchange between $I_{[u]_2}$ and $I_{[au]_2}$ effected by multiplying $X$ on the left by $A$, is also effected by multiplying $X$ on the right by $\omega$. In fact, we will have that $I_{[au]_2} = I_{[u]_2}\omega$ for all $a \in A$ and $u \in A^{\omega}$. From this it follows that $X\omega = A^{\omega}(I_{[u]_2})\omega \cong A^{\omega}(I_{[u]_2}\omega) \cong A^{\omega}(J_{[u]_2}) = Y$, as is our aim. The orders $I_{[u]_2}$ will be $\mathbb{Z}$-sums of \emph{surordinals}: orders all of whose non-trivial final segments are ordinals, but that are not ordinals themselves. 

The final step of the proof is to use Theorem \ref{fixedpoint} to prove that the orders $X$ and $Y$ we construct are non-isomorphic.

\theoremstyle{definition}
\newtheorem{sbbothsides}[theorem]{Theorem}
\begin{sbbothsides} \label{sbbothsides}
There exists a pair of non-isomorphic linear orders that are left-hand and right-hand divisors of one another. Specifically, there exist non-isomorphic orders $X$ and $Y$ that satisfy the four isomorphisms $X \cong AY$, $X \cong Y\omega$, $Y \cong AX$, and $Y \cong X\omega$, where $A=\omega_1^*+\omega_1$. 
\end{sbbothsides}
\begin{proof}

As just noted, we are left only to construct the orders $I_{[u]_2}$, show that $X \omega \cong Y$ and $Y \omega \cong X$, and prove $X \not\cong Y$. 

In what follows, $\omega^{\omega}$ has its traditional meaning as an ordinal, and not as the collection of $\omega$-length sequences on $\omega$. That is, $\omega^{\omega} =  \sup_{n < \omega} \omega^n = \omega + \omega^2 + \omega^3 + \ldots$ The ordinals $\omega^n$ also appear, though in this case there is no ambiguity in the notation, since $\omega^n$ (as an ordinal) is isomorphic to $\omega^n$ (as the collection of $n$-sequences on $\omega$, ordered lexicographically). In several places we will use the fact that if $\alpha$ is a strict initial segment of $\omega^{\omega}$, then $\alpha + \omega^{\omega} \cong \omega^{\omega}$. 

We first define, for every $i \in \mathbb{Z}$, an order $L_i$. For the non-negative integers, we let
\[
\begin{array}{l c l}
L_0 & = & \ldots + 3\omega^3 + 2\omega^2 + \omega + \omega^{\omega} \\
L_1 & = & \ldots + 3\omega^4+ 2\omega^3 + \omega^2 + \omega^{\omega} \\
& \vdots & \\
L_n & = & \ldots + 3\omega^{n+3} + 2\omega^{n+2} + \omega^{n+1} + \omega^{\omega} \\
& \vdots & 
\end{array}
\]

On the other side, let
\[
\begin{array}{l c l}
L_{-1} & = & \ldots + 4\omega^3 + 3\omega^2 + 2\omega + \omega^{\omega} \\
L_{-2} & = & \ldots + 5\omega^3 + 4\omega^2 + 3\omega + \omega^{\omega} \\
& \vdots & \\
L_{-n} & = & \ldots (n+3)\omega^3 + (n+2)\omega^2 + (n+1)\omega + \omega^{\omega} \\
& \vdots & 
\end{array}
\]

Observe that each of these orders is scattered, since they are $\omega^*$-sums of ordinals. We claim that for $i, j \in \mathbb{Z}$ with $i \neq j$ we have $L_i \not\cong L_j$. This follows from more general results due to Jullien \cite{Jullien} and independently Slater \cite{Slater}. We will give a proof using Slater's results below. However, for the sake of completeness, we sketch the argument that $L_0 \not\cong L_1$, since it is not too difficult and contains the essential ideas needed to prove more generally that the $L_i$ are pairwise non-isomorphic. 

For this, we introduce some local terminology. Suppose that $C = (I, J)$ is a cut in some linear order $L$. We say that $C$ has \emph{type} $n$ (where $n \in \omega$), and write $tp(C) = n$, if $I$ has a final segment isomorphic to $\omega^n$. ($C$ has type $0$ if $I$ has a top point.) In general, a cut $C$ need not have a type, but if $tp(C) = n$ then $C$ is not of type $m$ for any $m \neq n$. This is because, for each fixed $n$, the ordinal $\omega^n$ has the property that all of its nonempty final segments are isomorphic to $\omega^n$, and $\omega^n \not\cong \omega^m$ for $n \neq m$. Observe that in each $L_i$ defined above, if $C = (I, J)$ is a cut with neither $I$ nor $J$ empty, then $C$ has a type. 

Now, we have that $L_0 = \ldots + 3\omega^3 + 2\omega^2 + \omega + \omega^{\omega}$ and $L_1 = \ldots + 3\omega^4+ 2\omega^3 + \omega^2 + \omega^{\omega}$. Let us write these orders as longer sums by separating, for every $k \in \omega$, the copies of $\omega^k$ appearing in these representations. We write
\[
L_0 = \ldots + \omega^3 + \omega^3 + \omega^3 + \omega^2 + \omega^2 + \omega + \omega^{\omega},
\]
\[
L_1 = \ldots + \omega^4 + \omega^4 + \omega^4 + \omega^3 + \omega^3 + \omega^2 + \omega^{\omega}.
\]
The coinitial sequences of cuts at the + signs in these representations possess a particular recursive property that will allow us to distinguish these orders, as we now show.

Suppose that $L$ is a linear order in which every cut has a type. Suppose that $\ldots < C_2 < C_1 < C_0$ is a decreasing sequence of cuts in $L$, which we denote by $(C_k)_{k\in \omega}$, or simply $(C_k)$. We call such a sequence a \emph{ladder} if 
\begin{itemize} 
\item[(1)] $(C_k)_{k\in \omega}$ is coinitial in $L$,
\item[(2)] for all $k \in \omega $, we have $tp(C_{k+1}) \geq tp(C_k)$,
\item[(3)] for all  $k \in \omega$, if $D$ is a cut with $C_{k+1} < D < C_k$, then $tp(D) < tp(C_k)$. 
\end{itemize} 
Conditions (2) and (3) say that $C_{k+1}$ is the rightmost cut to the left of $C_k$ whose type is at least as large as the type of $C_k$. 

In a given order $L$, a ladder need not exist. However, if a ladder does exist, then it is an invariant of the order, in the sense that any two ladders eventually coalesce. That is, if $(C_k)$ and $(D_l)$ are ladders in $L$, then there exist $k_0, l_0 \in \omega$ such that for all $n \in \omega$ we have  $C_{k_0+n} = D_{l_0+n}$. 

To see this, let us first show that there must exist $k_0, l_0$ such that $C_{k_0} = D_{l_0}$. If not, the ladders are disjoint. But then, since they are both coinitial in $L$, there must be two consecutive cuts $C_{i+1} < C_i$ strictly between which lies at least one $D$-cut. Let $D_j$ be the leftmost such cut, so that $C_{i+1} < D_j < C_i$ and $D_{j+1} \leq C_{i+1}$. Then actually we must have have $D_{i+1} < C_{j+1}$ since the ladders are disjoint, so that $D_{j+1} < C_{i+1} < D_j < C_i$. Now, from conditions $(2)$ and $(3)$ in the definition of ladder, and from the inequality $C_{i+1} < D_j < C_i$, we obtain $tp(D_j) < tp(C_{i+1})$. But from condition (2) and $D_{j+1} < C_{i+1} < D_j$ we obtain $tp(C_{i+1}) < tp(D_j)$, a contradiction. 

Hence for some $k_0, l_0$ we have $C_{k_0} = D_{l_0}$.  But then it must be that $C_{k_0+1} = D_{l_0+1}$, since these are each the rightmost cuts to the left of $C_{k_0} = D_{l_0}$ of at least as large a type. Inductively we obtain $C_{k_0+n} = D_{l_0+n}$ for every $n$, as claimed.

We can express this fact in a different way. Fix a sequence $u = (n_1, n_2, \ldots)$, where $n_k \in \omega$ for all $k$. We say that a given ladder $(C_k)$ has \emph{spectrum} $u$ if $tp(C_k) = n_k$ for all $k \in \omega$. Suppose that in a given order, $(C_k)$ is a ladder with spectrum $v$ and $(D_l)$ is a ladder with spectrum $v'$. Then by what we have just proved, we must have that $v$ and $v'$ are tail-equivalent. (The appearance of tail-equivalence here is coincidental, and has nothing to do with its appearance elsewhere in this paper.)

Finally, returning to our context, we again write $L_0$ and $L_1$ as above:
\[
L_0 = \ldots + \omega^3 + \omega^3 + \omega^3 + \omega^2 + \omega^2 + \omega + \omega^{\omega},
\]
\[
L_1 = \ldots + \omega^4 + \omega^4 + \omega^4 + \omega^3 + \omega^3 + \omega^2 + \omega^{\omega}.
\]

The cuts at the + signs in this representation of $L_0$ form a ladder with spectrum $v = (1, 2, 2, 3, 3, 3, \ldots)$, and the cuts at the $+$ signs in $L_1$ form a ladder with spectrum $v'=(2, 3, 3, 4, 4, 4, \ldots)$. If $L_0$ and $L_1$ were isomorphic, then any ladder in $L_0$ and any ladder in $L_1$ would have tail-equivalent spectra. But $v \not\sim v'$, and so it must be that these orders are not isomorphic.  

Slater proves a more general result, and we argue from his paper to show that $L_i \not\cong L_j$ whenever $i \neq j$. Suppose that we have orders $L$ and $M$ such that 
\[
\begin{array}{lll}
L & = & \ldots + l_2\omega^{k_2} + l_1\omega^{k_1} + \omega^{\omega} \\
M & = & \ldots + l_2'\omega^{k_2'} + l_1'\omega^{k_1'} + \omega^{\omega},
\end{array}
\]
where the $l_n, l_n', k_n, k_n'$ are all positive integers, and furthermore $k_1 < k_2 < \ldots$ and $k_1' < k_2' < \ldots$ are strictly increasing sequences. In the terminology of Slater's paper, $L$ and $M$ are $RJ$ types of type 4 (see Theorem 2 of \cite{Slater}). By Theorem 4 of \cite{Slater}, if $L \cong M$, then there exists an $r \geq 0$ and $N$, such that for every $n \geq N$, we either have that $k_n' = k_{n+r}$ and $l_n' = l_{n+r}$, or we have that $k_n = k_{n+r}'$ and $l_n = l_{n+r}'$. That is, for $L$ and $M$ to be isomorphic, it is necessary that the coefficients $l_n, l_m'$ and exponents $k_n, k_m'$ eventually agree, up to some shift of index. 

If we compare $L_i$ and $L_j$ for $i \neq j$, we see that while these orders are of the same form as $L$ and $M$, they do not satisfy the condition necessary for their isomorphism. Hence $L_i \not\cong L_j$, as claimed. 

Though they are pairwise non-isomorphic, the $L_i$ are all closely related. Namely, we claim $L_i \omega = L_{i+1}$ for all $i \in \mathbb{Z}$. There are three cases to verify. If $i \geq 0$, we have
\[
\begin{array}{lll}
L_i \omega & = & (\ldots + 2\omega^{i+2} + \omega^{i+1} + \omega^{\omega})\omega \\
& \cong & \ldots + 2\omega^{i+2}\omega + \omega^{i+1}\omega + \omega^{\omega}\omega \\
& \cong & \ldots + 2\omega^{i+3} + \omega^{i+2} + \omega^{\omega} \\
& \cong & \ldots + 2\omega^{(i+1)+2} + \omega^{(i+1)+1} + \omega^{\omega} \\
& = & L_{i+1},
\end{array}
\]
where, in going from the second to third line, we have used the fact that $\omega^{\omega} \omega \cong \omega^{1+\omega}$ (reversing the exponent, as noted in Section 2.2) $\cong \omega^{\omega}$. 

For $i = -1$ we have
\[
\begin{array}{lll}
L_{-1}\omega & = & (\ldots + 3\omega^2 + 2\omega + \omega^{\omega}) \omega \\
& \cong & \ldots + 3\omega^3 + 2\omega^2 + \omega^{\omega} \\
& \cong & \ldots + 3\omega^3 + 2\omega^2 + \omega + \omega^{\omega} \\
& = & L_0.
\end{array}
\]
where, in going from the second to third line, we have used that $\omega + \omega^{\omega} \cong \omega^{\omega}$. Similarly, if $i < -1$, so that $i = -n$ for some $n > 1$ we have
\[
\begin{array}{lll}
L_i \omega & = & L_{-n}\omega \\
& = & (\ldots + (n+2)\omega^2 + (n+1)\omega + \omega^{\omega}) \omega \\
& \cong & \ldots + (n+2)\omega^3 + (n+1) \omega^2 + \omega^{\omega} \\
& \cong & \ldots + ((n-1)+3)\omega^3 + ((n-1)+2)\omega^2 + ((n-1)+1)\omega + \omega^{\omega} \\
& = & L_{-(n-1)} \\
& = & L_{i+1},
\end{array}
\]
where, in going from the third to fourth line, we use that $((n-1) + 1)\omega + \omega^{\omega} \cong \omega^{\omega}$. Hence $L_i\omega = L_{i+1}$ in all cases, as claimed. 

We are now almost ready to define the orders $I_{[u]_2}$ that will appear in the replacement $X = A^{\omega}(I_{[u]_2})$. These orders will each be one of the three orders $I_{even}, I_{odd}$, and $I$, defined as follows:
\[
\begin{array}{lll}
I_{even} & = & \ldots + L_{-2} + L_0 + L_2 + \ldots \\
I_{odd} & = & \ldots + L_{-1} + L_1 + L_3 + \ldots \\
I & = & \ldots + L_{-1} + L_0 + L_1 + \ldots
\end{array}
\]

Before defining the $I_{[u]_2}$, we prove that these three orders are pairwise non-isomorphic. Note first that for a given $i$, every cut in $L_i$ is either a $(1, 1)$-cut or $(\omega, 1)$-cut. The only cuts in the orders $I_{even}, I_{odd}$, and $I$ that do not fall in the midst of an $L_i$ occur at the $+$ signs in the above representations, and these cuts are $(\omega, \omega)$-cuts. Hence these are the only $(\omega, \omega)$-cuts appearing in these orders. 

Now suppose, for example, that there exists an isomorphism $f: I_{even} \rightarrow I_{odd}$. It must be, then, that $f[L_0] \subseteq L_k$ for some odd integer $k$. This is because $f[L_0]$ is an interval in $I_{odd}$, and every interval in $I_{odd}$ is either a subinterval of some $L_k$ or contains an $(\omega, \omega)$-gap. It cannot be that $f[L_0]$ contains an $(\omega, \omega)$-gap, since $L_0$ does not. But then we must actually have $f[L_0] = L_k$, since by a symmetric argument $f^{-1}[L_k]$ must be a subinterval of $L_m$ for some even integer $m$, and the only possible $m$ is $m=0$. 

This is a contradiction. It cannot be that $f[L_0] = L_k$ since this would mean that the orders $L_0$ and $L_k$ are isomorphic. But $L_0$ is never isomorphic to $L_k$ for $k$ odd. Hence $I_{even} \not\cong I_{odd}$. By similar arguments, $I_{even} \not\cong I$ and $I_{odd} \not\cong I$, as claimed. 

However, it is easy to see that we have $I_{even}\omega \cong I_{odd}$, $I_{odd}\omega \cong I_{even}$, and $I\omega \cong I$. For example, to verify the first isomorphism, we check
\[
\begin{array}{lll}
I_{even}\omega & = & (\ldots + L_{-2} + L_0 + L_2 + \ldots)\omega \\
& \cong & + \ldots + L_{-2}\omega + L_0\omega + L_2\omega + \ldots \\
& \cong & + \ldots + L_{-1} + L_1 + L_3 + \ldots \\ 
& = & I_{odd},
\end{array}
\]
and similarly for the other two isomorphisms. It follows that all three orders are invariant under right multiplication by $\omega^2$, that is $I_{even}\omega^2 \cong I_{even}, I_{odd}\omega^2 \cong I_{odd}$, and $I\omega^2 \cong I$. 

Now we can define the $I_{[u]_2}$. For every tail-equivalence class $C \subseteq A^{\omega}$, fix a representative $u_C$ (so that $C = [u_C]$). There are two cases. If $u_C \not\sim_2 au_C$, so that $[u_C]_2 \cap [au_C]_2 = \emptyset$, we let $I_{[u_C]_2} = I_{even}$ and $I_{[au_C]_2} = I_{odd}$. If $u_C \sim_2 au_C$, so that $[u_C]_2 = [au_C]_2 = [u_C]$, we let $I_{[u_C]_2} = I$. Then by above, we have that $I_{[u]_2}\omega \cong J_{[u]_2}$ for all $u \in A^{\omega}$. Depending on the $u$, this is just the isomorphism $I_{even}\omega \cong I_{odd}$, $I_{odd}\omega \cong I_{even}$, or $I\omega \cong I$. 

Let $X = A^{\omega}(I_{[u]_2})$, and let $Y = AX$. Then $Y \cong A^{\omega}(J_{[u]_2})$, where for every $u \in A^{\omega}$ and $a \in A$ we have $J_{[u]_2} = I_{[au]_2}$. From our remarks in Section 2.2, it is automatic that $X \cong AY$. We claim that these orders have the remaining desired properties, namely, that $X \cong Y\omega$, $Y \cong X\omega$, and $X \not\cong Y$. 

The first two are easy to verify. First, we have
\[
\begin{array}{lll}
X \omega & = & A^{\omega}(I_{[u]_2})\omega \\
& \cong & A^{\omega}(I_{[u]_2}\omega) \\
& \cong & A^{\omega}(J_{[u]_2}) \\
& = & Y.
\end{array}
\]
Likewise we may show $Y\omega \cong X$.

So it remains to prove $X \not \cong Y$. First note that since $A$ has no endpoints, the order $A^{\omega}$ is dense. Thus every interval of $A^{\omega}$ is dense. It follows that, in general, if $A^{\omega}(M_u)$ and $A^{\omega}(N_u)$ are replacements of $A^{\omega}$ with none of the $M_u, N_u$ empty, and $g: A^{\omega}(M_u) \rightarrow A^{\omega}(N_u)$ is an isomorphism, then for a given $u$ we must have that either $f[M_u] \subseteq N_v$ for some $v$, or that $f[M_u]$ (and hence $M_u$) contains a dense suborder. (The ``or" here is non-exclusive.)

Now, suppose that $f: X \rightarrow Y$ is an isomorphism. We view $f$ as an isomorphism of $A^{\omega}(I_{[u]_2})$ with $A^{\omega}(J_{[u]_2})$. None of the orders $L_k$ contains a dense suborder, and so neither do the $I_{[u]_2}$, $J_{[u]_2}$, as these are just $\mathbb{Z}$-sums of the $L_k$. By our observation above, it must be that for every $u \in A^{\omega}$ there is a $v$ such that $f[I_u] \subseteq J_v$. Conversely, for every $v \in A^{\omega}$ there must be a $u$ such that $f^{-1}[J_v] \subseteq I_u$. Combining these observations gives that in fact for every $u$ there is a $v$ such that $f[I_u] = J_v$. In particular, for such a pair $u, v$ we have that $I_u \cong J_v$ as linear orders. We will assume for convenience that $f$ is actually the identity on each $I_u$, that is, that $f((u, x)) = (v, x)$, since if $f$ ever acts non-trivially on the right coordinates we can replace $f$ with another isomorphism that does not, but still sends $I_u$ onto $J_v$. 

This means there is an automorphism $g: A^{\omega} \rightarrow A^{\omega}$ such that for every $u \in A^{\omega}$ and $x \in I_u$ we have $f((u, x)) = (g(u), x)$. By Theorem \ref{fixedpoint}, the automorphism $g$ has a fixed point $u = (\alpha_0, -\alpha_1, \alpha_2, -\alpha_3, \ldots)$, where the $\alpha_i$ are non-zero ordinals in $\omega_1$. For such a $u$ we have $u \not\sim_2 au$, and hence $I_{[u]_2} \not\cong I_{[au]_2}$: one of these orders is $I_{even}$, and the other is $I_{odd}$. Thus one of $I_u, J_u$ is $I_{even}$ and the other is $I_{odd}$. But then since $g$ fixes $u$ it must be that $f[I_u] = J_u$, a contradiction, as these orders are non-isomorphic. Hence $X \not\cong Y$, and the theorem follows. 
\end{proof}

\end{document}